\documentclass[10pt]{amsart}
\usepackage{amssymb,amsmath,amsthm,amsfonts,amsopn,url}
\usepackage{amscd,amssymb,amsopn,amsmath,amsthm,graphics,amsfonts,enumerate,verbatim,calc}
\usepackage[all]{xy}
\theoremstyle{plain}
\newtheorem{thm}{Theorem}[section]
\newtheorem*{mt*}{Theorem (Huneke-Lyubeznik, [3], Theorem 2.9)}
\newtheorem*{cj*}{Conjecture}
\newtheorem*{nt*}{Notations}

\newtheorem{lemma}[thm]{Lemma}
\newtheorem{cor}{Corollary}
\newtheorem{rem}{Remark}

\theoremstyle{definition}
\newtheorem*{definition}{Definition }
\newcommand{\ideal}[1]{\mathfrak{#1}}

\newcommand{\m}{\ideal{m}}
\newcommand{\n}{\ideal{n}}
\newcommand{\p}{\ideal{p}}
\newcommand{\q}{\ideal{q}}

\newcommand{\func}[1]{\mathrm{#1} \,}

\newcommand{\Hom}{\func{Hom}}
\newcommand{\Ass}{\func{Ass}}

\newcommand{\rad}{\func{rad}}

\title[]{Eisenstein extension, connectedness and the second vanishing theorem}

\author[]{Rajsekhar Bhattacharyya}

\address{Dinabandhu Andrews College, Garia, Kolkata 700084, India}

\email{rbhattacharyya@gmail.com}

\thanks{}


\begin{document}

\begin{abstract}
In this paper, at first, we show that for a ramified regular local ring $S$, which is an Eisenstein extension of an unramified regular local ring $R$, when an ideal $I$ of $S$ is extended from an ideal $J$ of $R$, the punctured spectrum of $R/J$ is connnected if that of $S/JS$ is connected. Using this, we extend the result of SVT to complete ramified regular local ring only for the extended ideals. If the punctured spectrum of $S/JS$ is disconnected then that of $R/J$ is also disconnected when every minimal primes $\p$ of $J$, $R/\p$ is normal. Under this situation we prove that both of them have the same number of connected components. Finally, we show that for both unramified and ramified regular local rings (for extended ideal via Eisenstein extension), two top-most local cohomology modules satisfy the Conjecture 1 of \cite{L-Y}, although the conjecture is false in general.  
\end{abstract}

\maketitle
\section{introduction}

Recall that the cohomological dimension of an ideal $J$ of a Noetherian ring $R$ is the maximum index $i \geq0$ for which the local cohomology module $H^i_J(R)$ is nonzero. The cohomological dimension of the maximal ideal of a local ring coincides with the ring's dimension \cite{HartshorneLC}.  

In this context we mention Hartshorne-Lichtenbaum vanishing theorem or ``HLVT'' \cite{HartshorneCD}. It states that: For any complete local domain $R$ of dimension $d$, $H^d_J(R)$ vanishes if and only if $\dim(R/J)>0$. One may regard the HLVT as the ``first vanishing theorem'' for local cohomology.

If the ring $R$ contains a field, the ``second vanishing theorem'' or ``SVT'' of local cohomology states the following: Let $R$ be a complete regular local ring of dimension $d$ with a separably closed residue field, which it contains. Let $J\subseteq R$ be an ideal such that $\dim(R/J)\geq 2$. Then $H^{d-1}_J(R)=0$  if and only if the punctured spectrum of $R/J$ is connected \cite{HartshorneCD,H-L,Ogus,P-S}. 

If regular local ring does not contain fields, then, recently in \cite{Zh} (see Theorem 1.6), a version of SVT is proved:

Let $(R, \m)$ be a $d$-dimensional complete unramified regular local ring of mixed characteristic, whose residue field is separably closed. Let $I$ be an ideal of $R$, then $H^j_{I}(R)=0$ for all $j\geq d-1$ if and only if $\dim(R/J)\geq 2$ and the punctured spectrum of $R/J$ is connected. 

Previously, in \cite{CohDim} (see Theorem 3.8), a slightly different version the SVT has been proved for unramified regular local ring of mixed characteristic which follows as special case from the above defined version: 

Let $(R, \m)$ be a $d$-dimensional complete unramified regular local ring of mixed characteristic, whose residue field is separably closed. Let $J$ be an ideal of $R$ for which $\dim(R/\p) \geq 3$ for every minimal prime $\p$ of $J$. Then $H^{d-1}_J(R)=0$ if and only if the punctured spectrum of $R/J$ is connected. It is to be noted that this version uses only the ideals $J\subset R$ such that $R/J$ is equidimensional.

In this paper, at first, we show that for a ramified regular local ring $S$, which is an Eisenstein extension of an unramified regular local ring $R$, when an ideal $I$ of $S$ is extended from an ideal $J$ of $R$, the punctured spectrum of $R/J$ is connnected if that of $S/JS$ is connected. Using this, we extend the result of the SVT to complete ramified regular local rings only for the extended ideals (see Theorem 2.2): 

Let $(S, \n)$ be a $d$-dimensional complete ramified regular local ring of mixed characteristic, with a separably closed residue field and $I$ be an ideal of it. Assume $I= JS$ is an extension of an ideal $J$ of any unramified complete regular local ring $R$ such that $S$ is an Eisenstein extension of $R$. In this paper, at first, we show that for such a ramified regular local ring $S$, the punctured spectrum of $R/J$ is connnected if that of $S/JS$ is connected. Using this, we extend the result of the SVT to complete ramified regular local rings only for the extended ideals (see Theorem 2.2): Let $\dim(S/I) \geq 2$, then $H^{d-1}_I(S)=0$ if and only if the punctured spectrum of $S/I$ is connected. Moreover, if the punctured spectrum of $S/JS$ is disconnected then that of $R/J$ is also disconnected when every minimal primes $\p$ of $J$, $R/\p$ is normal and both of them have the same number of connected components (see Theorem 3.1).

In \cite{L-Y}, in Conjecture 1 it is claimed that for Noeherian regular local ring $R$ containing an ideal $I$, if $H^{i}_{I}(R)\neq 0$ then $0\in\Ass (D_R (H^{i}_{I}(R)))$, where $D_R()$ is Matlis dula functor. This conjecture is proved to be false in mixed characteristics \cite{DSZ}, in fact they show that for power series ring of dimension six over DVR 4th local cohomology violates the conjecture. But, from the results of Theorem 3.1 we show that for both unramified and ramified regular local rings (for extended ideal via Eisenstein extension), two top-most local cohomology modules satisfy the conjecture (see Corollary 1). Moreover, they also provide examples of local cohomology modules whose support of matis dual and spectrum of the rings coincide, which is stated in Corollary 1.2 of \cite{L-Y}.

\section{the main result}

Let $(R;m)$ be a local ring of mixed characteristic $p > 0$. We say $R$ is unramified if $p\notin {\m}^2$ and it is ramified if $p\in {\m}^2$. For normal local ring $(R,\m)$, consider the extension ring defined by $S=R[X]/f(X)$ where $f(X)=X^n+ a_1X^{n-1}+\ldots +a_n$ with $a_i\in \m$ for every $i=1,\ldots,n$ and $a_n\notin {\m}^2$. This ring $(S,\n)$ is local and it is defined as an Eisenstein extension of $R$ and $f(X)$ is known as an Eisenstein polynomial (see page 228-229 of \cite{Matsumura}). 

Here we note down the following important results regarding Eisenstein extensions:

(1) An Eisenstein extension of regular local ring is regular local (see Theorem 29.8 (i) of \cite{Matsumura}) and in this context, we observe that an Eisenstein extensions of a regular local ring is again a regular local ring.

(2) Every ramified regular local ring is an Eisenstein extension of some unramified regular local ring (see Theorem 29.8 (ii) of \cite{Matsumura}. 

(3) The Eisenstein extension mentioned in (1) and (2) are faithfully flat (apply Theorem 23.1 of \cite{Matsumura}) and it is also a finite extension (see Lemma 1 in p-228 of \cite{Matsumura}).

Before presenting the main result, we recollect the defintion of Huneke-Lyubeznik graph.

\begin{definition}
For a local ring $(R,\m)$, let $\p_1,\ldots,\p_t$ be the minimal primes. Then the graph $\Theta_R$ is defined for $t$ vertices labeled by $1,\ldots t$, such that there will be an edge between two different vertices $i$ and $j$, if $\p_i + \p_j$ is not $\m$-primary.
\end{definition}

\begin{mt*}
Given a complete local ring $R$, the graph $\Theta_R$ is
connected if and only if the punctured spectrum of $R$ is connected.
\end{mt*}

Now, we observe the following Lemma.

\begin{lemma}
Let $(R, \m)$ be a $d$-dimensional complete unramified regular local ring of mixed characteristic. Let $(S, \n)$ be a ramified complete regular local ring obtained via Eisenstein extension of $R$ where $f(X)\in R[X]$ is the Eisenstein polynomial.\\
(1) If the residue field of $S$ is separably closed then so is the residue field of $R$\\
(2) If the image of an Eisenstein polynomial $f(X)\in R[X]$ is prime in the ring $(R/\p)[X]$, then $\p S$ is also prime in $S$.\\
(3) For ideal $J\subset R$, $\dim (S/JS)=\dim (R/J)$.\\
(4) For $I= JS$, if the punctured spectrum of $S/I$ is connected then that of $R/J$ is also connected. 
\end{lemma}

\begin{proof}
(1) In $S$, $\n= \m S +XS$. This gives $S/\n= (R[X]/f)/(\m+X)(R[X]/f)= (R[X]/\m R[X])/(X)(R[X]/\m R[X])= R/\m$, since $f-X^n \in \m R[X]$. So the result follows.

(2) Observe that for $\q=\p S$, $S/\q = S/\p S= (R[X]/fR[X])/\p R[X](R[X]/fR[X])= (R/\p)[X]/f(R/\p)[X]$. From the assertion, $(R/\p)[X]/f(R/\p)[X]$ is a domain, hence $\q\subset S$ is a prime. 

(3) Due to Eisenstein extension, $S$ is integral over $R$, hence $S/JS$ is integral over $R/J$ and they have same dimensions.

(4) Let $\q\subset S$ be the minimal prime of $I= JS$. Now $\p= \q\cap R$ is a minimal prime ideal of $J$, otherwise if $J\subset \p'\subset \p$, then by going down theorem there exists $\q'\subset S$ such that $JS\subset \q'\subset \q$. Conversely, for any minimal prime $\p$ of $J$, we claim that any $\q$ that lying over $p$ is a minimal prime of $JS$ as well as of $\p S$. Since $S$ is integral over $R$, such $\q$ always exists. To see this, suppose there exists some $JS\subset \q'\subset \q$. This gives $J\subset \q'\cap R\subset \p$. But, $\p$ is minimal over $R$ hence by lying over theorem, $\q= \q'$. So, for $t'\geq t$, if $\p_1,\ldots,\p_t$ are the minimal primes of $J$ and $\q_1,\ldots,\q_t'$ are the minimal primes of $JS$ then $\q_i\cap R= \p_i$ and $\q_i$ is minimal prime over $\p_i S$ as well for every $i=1,\ldots,t$. 

Now we assert that the punctured spectrum of $R/J$ is connected or equivalently the graph $\Theta_{R/J}$ of $\p_1,\ldots, \p_t$ is connected \cite{H-L}, Theorem 2.9. From the hypothesis, applying \cite{H-L}, Theorem 2.9 once again we get that, the puntured spectrum of $S/I$ is connected or equivalently the graph $\Theta_{S/I}$ of $\q_1,\ldots, \q_t'$ is connected, i.e for any pair of $\q_i, q_j$, $\q_i + \q_j$ is not $\n$-primary. So we would like to show for any pair of $\p_i, \p_j$, $\p_i + \p_j$ is not $\m$-primary. To see this, assume for some pair $\rad (\p_i + \p_j)= \m$. Since $S$ is an Eisenstein extension of $R$, we have $S/\m S= (R/\m)[X]/(X^n)$. This gives $\dim (S/\m S)= \dim ((R/\m)[X]/(X^n))= 0$. Thus $\rad (\m S)= \n$. Now $\rad (\q_i +\q_j)\supset \rad (\p_i S+ \p_j S)\supset \rad ((\p_i +\p_j)S)\supset (\rad (\p_i +\p_j))S\supset\m S$. This gives $\rad (\q_i +\q_j)\supset \rad (\m S)= \n$. This gives that the graph $\Theta_{S/I}$ of $\q_1,\ldots, \q_t'$ is not connected. 

\end{proof}

Now we prove the SVT over complete ramified regular local rings of mixed characteristic, only for the extended ideals.  

\begin{thm}
Let $(S, \n)$ be a $d$-dimensional complete ramified regular local ring of mixed characteristic, with a separably closed residue field and $I$ be an ideal of it. Assume $I= JS$ is an extension of an ideal $J$ of any unramified complete regular local ring $R$ such that $S$ is an Eisenstein extension of $R$. If $\dim(S/I) \geq 2$, then $H^{d-1}_I(S)=0$ if and only if the punctured spectrum of $S/I$ is connected. 
\end{thm}

\begin{proof}
We first assume $H^{d-1}_I(S)=0$. Suppose, the punctured spectrum of $S/I$ is disconnected.
Let $\n$ denote the maximal ideal of $S$, so that there exist ideals $I_1$ and $I_2$ of $S$ that are not $\n$-primary
for which $\rad (I_1\cap I_2)=\rad I$ and $\rad (I_1+I_2)=\n$. 
Consider the Mayer-Vietoris sequence,
\[
\cdots \to H^{d-1}_I(S)\to H^d_{I_1+I_2}(S)\to H^d_{I_1}(S)\oplus H^d_{I_2}(S)\to H^d_{I}(S)\to 0.
\]
Now $H^{d-1}_I(S)=0$, and $H^d_{I_1}(S)=H^n_{I_2}(S)=H^d_{I}(S)=0$
by the HLVT. Then $H^d_\n(S)=H^d_{I_1+I_2}(S)=0$, which contradics the HLVT. 

It should be mentioned that the part of the proof in the above paragraph, is same to that of \cite{CohDim}, Theorem 3.8, but for the sake of completeness we keep it here.

To prove the other direction, we can proceed as follows: From (3) of Lemma 2.1 we get that $\dim (R/J)\geq 2$ and from (4) of Lemma 2.1, we get punctured spectrum of $R/J$ is also connected. Moreover from (1) of above Lemma 2.1, we have that $R$ also has separably closed residue field. So, all the conditions on ramified ring $S$ and on the ideal $JS$ reduces to those on unramified ring $R$ and on the ideal $J$. So, by \cite{Zh}, Theorem 1.4, we get $H^{d-1}_{J}(R)= 0$. Since $S$ is flat over $R$, $H^{d-1}_{JS}(S)= H^{d-1}_{J}(R)\otimes S= 0$. This finishes the proof.
\end{proof}


\section{when punctured spectrum is not connected}

For Noetherian local ring $(R,\m)$, let $E_R$ be the $R$-injective hull of the residue field. Then for any $R$-module $M$, we set $D_R (M)= \Hom_R(M,E_R)$ as the Matlis dual of M. In the previous section we show that for an Eisenstein extension $S$ of $R$, when an ideal $I$ of a ramified regular local $S$ is extended from an ideal $J$ of an unramified regular local ring $R$, the punctured spectrum of $R/J$ is connnected if that of $S/JS$ is connected. In this section, in Theorem 3.1, we show that if punctured spectrum of $S/JS$ is disconnected then that of $R/J$ is also disconnected when every minimal primes $\p$ of $J$, $R/\p$ is normal. Under this situation we prove that both of them have the same number of connected components. 

From the results of Theorem 3.1 we show that for both unramified and ramified regular local rings (for extended ideal via Eisenstein extension), two top-most local cohomology modules satisfy the Conjecture 1 of \cite{L-Y} (see Corollary 1). Although, this conjecture is proved to be false in mixed characteristics \cite{DSZ}, in fact they show that for power series ring of dimension six over DVR, 4th local cohomology violates the conjecture.  

\begin{thm}  
Let $(S, \n)$ be a $d$-dimensional complete ramified regular local ring of mixed characteristic, with a separably closed residue field and $I$ be an ideal of it. Assume $I= JS$ is an extension of an ideal $J$ of any unramified complete regular local ring $(R, \m)$ such that $S$ is an Eisenstein extension of $R$ and for every minimal prime $\p$ of $J$, $R/\p$ is normal. \\
(a) Under above situation, the punctured spectrum of $R/J$ is not connnected if and only if that of $S/JS$ is not connected. Moreover, both of them have same number of connected components.\\
(b) If $\dim(S/JS) \geq 2$ and the punctured spectrum of $S/JS$ has $t$ connected components, then\\ 
(1) $D_R (H^{d-1}_J(R))= R^{\oplus t-1}$,\\
(2) $D_R (H^{d-1}_I(S))= S^{\oplus t-1}$,\\ 
(3) $D_R (H^{d-1}_J(R/fR))= D_{R/fR} (H^{d-1}_{J(R/fR)}(R/fR))$\\
= $D_R (H^{d-1}_J(S/XS))= D_{R/fR} (H^{d-1}_{J(S/XS)}(S/XS)) \subset R^{\oplus t-1}$.
\end{thm}  

\begin{proof}
(a) For ideal $J\subset R$, let $\p_1,\ldots,\p_t$ be the minimal primes of $J$. Consider the collection of ideals $\{ \p_1,\ldots,\p_t,\p_1+\ldots+\p_t+\m^2 \}$ and using prime avoidance(since we can choose at the most two non prime ideals in the collection to avoid) we can choose $a_1,\ldots, a_n$ which are not in any $\p_1,\ldots,\p_t$ along with $\p_1+\ldots+\p_t+\m^2$. Now image of any $a_i$ in $R/\p_j$is non zero and is in $\m/\p_j$ and moreover $a_n$ (infact every $a_i$ for $i=1$ to $n$) are also not in $\m^2(R/\p_j)$. 

We choose an Eisenstein polynomial $f(X)=X^n+ a_1X^{n-1}+\ldots +a_n$ and $S$ is obtained via Eisenstein extension of $R$, i.e. $S=R[X]/f(X)$. From hypothesis and using Lemma 1 in page 228 of \cite{Matsumura}, the image of $f(x)$ is prime in every $(R/\p_i)[X]$ for all $i= 1,\ldots,t$. Moreover, from (3) of Lemma 2.1, we get that $\p_i S$ is a prime in $S$ for all $i= 1,\ldots,t$. 

Since $\p_i S$'s are primes and from the proof of (4) of Lemma 2.1 We find that they are the minimal primes of $JS$. We set $\q_i= \p_i S$ for every $i=1,\ldots,t$. 
Thus there is a one to one correspondence between the minimal primes over $J$ and those over $JS$, and if $\p_1,\ldots,\p_t$ be the minimal primes of $J$ and $\q_1,\ldots,\q_t$ be the minimal primes of $JS$ then $\q_i= \p_i S$ for every $i=1,\ldots,t$. 

From (4) of Lemma 2.1, we know that if the punctured spectrum of $S/JS$ is connected then the punctured spectrum of $R/J$ is also connected. Now we assert the converse, i.e. if the punctured spectrum of $R/J$ is connected then so is the puntured spectrum of $S/JS$ or equivalently the graph $\Theta_{S/JS}$ of $\q_1,\ldots, \q_t$ is connected, i.e for any pair of $\q_i, \q_j$, $\q_i + \q_j$ is not $\n$-primary \cite{H-L}, Theorem 2.9. To see this, assume otherwise that for some pair $\q_i$ and $\q_j$, $\rad (\q_i + \q_j)= \n$. Since $\q_i= \p_i S$ and $\q_j= \p_j S$, we get $\m= \n \cap R= \rad (\p_i S + \p_j S)\cap R= (\rad (\p_i + \p_j)S)\cap R= \rad ((\p_i + \p_j)S\cap R)= \rad (\p_i + \p_j)$. This is a contradiction. Thus, we conclude that $S/JS$ is connected. Thus we find the number of connected components in the punctured spectrum of $R/J$ and that of $S/JS$ are same.

(b) From results of Lemma 2.1, (a) of this Theorem, Theorem 1.6 of \cite{Zh} and using proof of Proposition 3.12 of \cite{CohDim}, we get $H^{d-1}_{J}(R)= E_R^{\oplus (t-1)}$, where $E_R$ is the $R$-injective hull of its residue field. If we apply the Matlis duality functor $D_R (-)$ on $H^{d-1}_{J}(R)$, we get, $D_R (H^{d-1}_{J}(R))= R^{\oplus (t-1)}$. Thus we prove (1). 

For (2), we can argue in the following way: Since $S$ is finite and flat over local ring $R$, it is finite and free, thus $D_R (H^{d-1}_{I}(S))= D_R (H^{d-1}_{J}(R)\otimes S)= D_R (H^{d-1}_{J}(R))\otimes S= S^{\oplus (t-1)}$. 

Set $f(X)= f$. Now, $(X, f)$ as well as $(f,X)$ are two $R[X]$-regular sequences (for polynomials $g, h\in R[X]$, if $fh= Xg$ then $X|h$ and $f|g$, since R[X] is a UFD). This gives the following commutative diagram of short exact sequences of $R[X]$-modules whose rows and columns are exact. 

\[
\CD
@. 0@. 0@. 0@. @. \\
@. @VVV @VVV @VVV @.\\
0 @>>>R[X]@>f>>R[X]@>>>R[X]/fR[X]@>>>0 @.\\
@. @VX VV @VVX V @VVX V @. \\
0 @>>>R[X]@>f>>R[X]@>>>R[X]/fR[X]@>>>0 @.\\
@. @VVV @VVV @VVV @. \\
0@>>>R[X]/XR[X]@>f>>R[X]/XR[X]@>>>R[X]/(X,f)R[X]@>>>0 @.\\
@. @VVV @VVV @VVV @.\\
@. 0@. 0@. 0@. @. 
\endCD
\]

The above diagram gives long exact sequences of $R[X]$-modules in both horizontal and vertical direction. Application of the Matlis duality functor $D_R (-)$ gives the following diagram of long exact sequences of $R$-modules where all the rows and columns are exact. Here, it is to be noted that for an $R$-module $M$, $D_R(M[X])$ is also an $R[X]$-module, but $D_R (-)$ is not a functor in the category of $R[X]$-module or $R[[X]]$-module and for that reason neither $D_R (f)$ is a map multiplication by "$f$" nor $D_R (X)$ is a map multiplication by "$X$".

\[
\CD
@.@AAA @AAA @AAA @AAA\\
@. D_R (H^{i-1}_{J R[X]}(R[X])) @<D_R (f)<<D_R (H^{i-1}_{J R[X]}(R[X]))@<<<D_R (H^{i-1}_{JS}(S))@<<<D_R (H^i_{J R[X]}(R[X]))\\
@. @AD_R(X) AA@AD_R(X) AA @AD_R(X) AA @AD_R(X) AA \\
@. D_R (H^{i-1}_{J R[X]}(R[X])) @<D_R (f)<<D_R (H^{i-1}_{J R[X]}(R[X]))@<<<D_R (H^{i-1}_{JS}(S))@<<<D_R (H^i_{J R[X]}(R[X]))\\
@. @AAA @AAA @AAA @AAA @. \\
@. D_R (H^{i-1}_{J}(R)) @<D_R (f)<<D_R (H^{i-1}_{J}(R))@<<<D_R (H^{i-1}_{J}(\frac{R[X]}{(X,f)R[X]}))@<<<D_R (H^{i}_{J}(R)) @. @.\\
@. @AAA @AAA @AAA @AAA @. \\
@. D_R (H^{i}_{J R[X]}(R[X])) @<D_R (f)<<D_R (H^{i}_{J R[X]}(R[X]))@<<<D_R (H^{i}_{JS}(S))@<<<D_R (H^{i+1}_{J R[X]}(R[X])) @. @.\\
@. @AAA @AAA @AAA @AAA @. 
\endCD
\]

From the above result we observe that $\frac{R[X]}{(X,f)R[X]}= R/fR= S/XS$. Using the above diagram once again we get that $D_R (H^{d-1}_J(R/fR))= D_R (H^{d-1}_J(S/XS))\subset D_R (H^{d-1}_{J}(R))= R^{\oplus t-1}$. Now, for any $\phi\in D_R (H^{d-1}_J(R/fR)$, $f\phi= 0$ and also $\phi$ maps into $(0:_{E_R}f)\subset E_R$ and it is well known that $(0:_{E_R}f)= E_{R/fR}$. Thus $D_R (H^{d-1}_J(R/fR))= D_{R/fR} (H^{d-1}_{J(R/fR)}(R/fR))$ and similar is true for the ring $S/XS$. This finishes the proof of (3) and (4).

\end{proof}

\begin{rem}
As an example of primes $\p$ in $R$ such that $R/\p$ is normal as mentioned in the hypothesis of Theorem 3.1, consider $\p$ generated by part of regular system of parameters of $R$. Clearly $R/\p$ is regular local ring, hence normal.
\end{rem}

In the following corollary we show that Conjecture 1 of \cite{L-Y} is true in unramified and ramified situation for two top-most local cohomology modules. 

\begin{cor}
Let $S$ be a $d$-dimensional complete ramified regular local ring of mixed characteristic and it is an Eisenstein extension of an unramified regular local ring $R$ via Eisenstein extension. Let $J$ be an ideal of $R$. Assume $S$, $R$ and $J$ satisfy hypothesis of Theorem 3.1. Then for $J\subset R$ and for $JS\subset S$,\\
(a) Two top-most local cohomology modules for $R$ and $S$ satisfy Conjecture 1 of \cite{L-Y}.\\
(b) For each of them, the Support of the Matlis dual modules is the whole spectrum of the ring, which is stated in Corollary 1.2 of \cite{L-Y}\\
\end{cor}  

\begin{proof}
(a) We only need to show that $H^{d-1}_J(R)$ and $H^{d-1}_{JS}(S)$ satisfy the conjecture, since top nonzero local cohomology module is the injective hull of residue field. For $H^{d-1}_J(R)$, the proof follows immediately from (1) of (b) of Theorem 3.1. From (2) of (b) of Theorem 3.1, $D_R(H^{d-1}_{JS}(S))=S^{\oplus t-1}$, Now using adjointness of Hom-Tensor funtors and $Hom_R(S,E_R)= E_S$, for any $S$-module $M$, one can show that $D_R(M)\cong D_S(M)$ for $S$-module $M$ where it is viewed as an $R$-module via restriction. Thus $D_S (H^{d-1}_{JS}(S))= S^{\oplus (t-1)}$ and the result follows.

(b) Since $D_R(H^{d-1}_J(R))=R^{\oplus t-1}$ and $D_S(H^{d-1}_{JS}(S))=S^{\oplus t-1}$ the result follows immediately.
\end{proof}


\begin{rem}
In Theorem 3.1, we use $D_R (X)$, here we show that $D_R (X)$ is not a map multiplication by "$X$": Consider $R$-module $M$ and $R[X]$-module $N$ where $X$ is an indeterminate. We give an example of an $R$-linear map $\chi: M[X]\rightarrow N$ such that $D_R (X)(\chi)\neq X\chi$. To show this it is sufficient to construct a map $\chi$ such that $\chi(Xg(X))\neq X\chi(g(X))$ for some $g(X)=a_0 + a_1 X \ldots +a_n X^n \in M[X]$. Let $\psi:M\rightarrow N$ be an $R$-linear map and define $\chi(g(X))= (\frac{d}{dX}(\psi(a_0) + \psi(a_1) X \ldots +\psi(a_n) X^n)|_{X=0}$. Suppose for this $g(X)$, $\psi(a_0)\neq 0$ and $\psi(a_1)= 0$. Clearly $\chi$ is an $R$-linear map, but $(X\chi)(g(X))=X(\chi(g(X)))= X.0$ while $\chi(Xg(X))= (\frac{d}{dX}(\psi(a_0)X + \ldots +\psi(a_n) X^{n+1})|_{X=0}= \psi(a_0)\neq 0$. 

Similarly, choosing suitable choice of $g(X)$ one can prove that $\chi(fg(X))\neq f\chi(g(X))$ for some $g(X)$ and this proves $D_R (f)$ is not a map multiplication by "$f$". 
\end{rem}



\end{document}